\documentclass[11pt]{amsart}
\usepackage{amscd}
\usepackage{comment}
\usepackage{amsmath, amssymb, amsgen, amsthm, amscd, xspace, color}
\newcommand{\pf}{\begin{proof}}
\newcommand{\epf}{\end{proof}}
\newcommand{\eq}{\begin{equation}}
\newcommand{\eeq}{\end{equation}}
\newcommand{\eqn}{\begin{equation*}}
\newcommand{\eeqn}{\end{equation*}}

\newtheorem{theorem}[equation]{Theorem}

\newtheorem{prop}[equation]{Proposition}
\newtheorem{lemma}[equation]{Lemma}

\theoremstyle{remark}
\newtheorem{remark}[equation]{Remark}

\theoremstyle{definition}
\newtheorem{definition}[equation]{Definition}

\numberwithin{equation}{section}
\allowdisplaybreaks[1]
\begin{document}

\title{A new method to prove the irreducibility of the Eigenspace Representations for $\mathbb{R}^{n}$ semidirect with finite pseudo-reflection group}

\author{Jingzhe Xu}
\address[Xu]{Department of Mathematics, Hong Kong University of Science and technology,
Clear Water Bay, Kowloon, Hong Kong SAR, China}
\email{jxuad@ust.hk}
\abstract{We show that the Eigenspace Representations for $\mathbb{R}^{n}$ semidirect with a finite pseudo-reflection
group $K$, which satisfy some generic property are equivalent to the induced representations from
$\mathbb{R}^{n}$ to $\mathbb{R}^{n} \rtimes K$, which satisfy the same property by Mackey little group method.And the proof of the equivalence is by using
matrix coefficients and invariant theory.As a consequence, these eigenspace representations are irreducible. }
\endabstract

%
\maketitle     
%
\section{Introduction}
Eigenspace Representations for Riemannian symmetric spaces play a crucial rule in the research both of Group representations and
Geometric Analysis. Many mathematicians have a great interest in the conditions for which the eigenspace representations are
irreducible. In fact, Robert Steinberg[S1] has made a great contribution to it, by using purely invariant theory of finite reflection groups. However, our approach  is by using both of Mackey little group method and invariant theory. What's more, the main part of our proof is just accourding to a monomorphism. Therefore, it makes the proof much simpler. Also, Sigurdur Helgason is another master, who has classified the irreducibility about the eigenspace representations
with respect to both of the compact type and the noncompact type [H1] of the symmetric space $G/K$. What's more, when $G=\mathbb{R}^{n} \rtimes O(n)$, $K=O(n)$, the eigenspace about $G/K\cong \mathbb{R}^{n}$ becomes $\varepsilon_{\lambda}(\mathbb{R}^{n})=\{f\in \varepsilon(\mathbb{R}^{n})\mid Lf=-\lambda^{2}f\}$, where
$L$ denotes the ususal Laplacian on $\mathbb{R}^{n}$. And Helgason also proved the natural action of $G$ on $\varepsilon_{\lambda}(\mathbb{R}^{n})$ is
irreducible if and only if $\lambda\neq 0$[H2]. In this paper, we prove a more general result. We find a generic property under which the eigenspace
representations for $G=\mathbb{R}^{n} \rtimes K$, where $K$ is a finite pseudo-reflection group are irreducible.

\begin{definition}
finite pseudo-reflection groups

   if $K$ is a finite subgroup of $GL(\mathbb{R}^{n})$ and $M\in K$, then $M$ is called a pseudo-reflection if precisely one eigenvalue of $M$
is not equal to one. We call $K$ is a finite pseudo-reflection group if $K$ is a finite subgroup of $O(n)$, which is generated by pseudo-reflections.
\end{definition}
    The main idea of the proof is as follows: On one hand, we use Mackey little group method to get induced representations of $G$, which are irreducible.
On the other hand, for any such representation $Ind_{\mathbb{R}^{n}}^{\mathbb{R}^{n} \rtimes K}\chi$, which satisfies the generic property, we use matrix coefficient method to
get a monomorphism from it to some eigenspace representation$(T_{\lambda},E_{\lambda})$, which satisfies the same property. The amazing part is we know
dim$Ind_{\mathbb{R}^{n}}^{\mathbb{R}^{n} \rtimes K}\chi$=$\mid K\mid$ and will prove dim$E_{\lambda}\leq \mid K\mid$ by using invariant theory. Therefore, $Ind_{\mathbb{R}^{n}}^{\mathbb{R}^{n} \rtimes K}\chi$
is equivalent to $(T_{\lambda},E_{\lambda})$ and $(T_{\lambda},E_{\lambda})$ is irreducible.

   This paper is as follows. In Sect.2, we give the definitions of semidirect product of $G=\mathbb{R}^{n} \rtimes K$, eigenspace representations of $G$ and the generic property.
What's more, we recall the basic knowledge of Mackey litte group method. In Sect.3,we show there is a monomorphism from the unitary irreducible induced representations of $G$ to some eigenspace
representations both of which satisfy the generic property. In Sect.4, we show the dimension of any eigenspace of $G=\mathbb{R}^{n} \rtimes K$ is smaller or equal to the order of
$K$ by using invariant theory. Thus,the eigenspace representations, which satisfy the generic property are irreducible. In Sect.5, we illustrate the circumstances for a particular example when $K$ is equal to the Dihendral group $D_n$.

%
\section{Notation and Preliminaries}
\begin{definition}
semidirect product

Let $K$ be a finite pseudo-reflection group,

Define
\begin{equation*}
\varphi:K\rightarrow Aut(\mathbb{R}^{n})
\end{equation*}
\begin{equation*}
A\mapsto A
\end{equation*}
Then

\begin{equation*}
\begin{split}
&(x_1,k_1)\rtimes (x_2,k_2)=(x_1+k_1\cdot x_2,k_1\cdot k_2)\\
&(x,k)^{-1}=(-k^{-1}\cdot x,k^{-1}).
\end{split}
\end{equation*}
Under this definition,we have the identification
\begin{equation*}
\mathbb{R}^{n} \rtimes K/K\cong \mathbb{R}^{n}
\end{equation*}
\end{definition}

\begin{definition}
Eigenspace Representations

Let $G=\mathbb{R}^{n} \rtimes K$, where $K$ is a finite pseudo-reflection group. $D(G/K)$ is the algebra of all differential operators on $G/K$
which are invariant under $K$. For each homomorphism \ $\chi:D(G/K)\rightarrow \mathbb{C}$, consider the joint eigenspace
\begin{equation*}
E_{\chi}=\{f\in C^{\infty}(G/K)\mid Df=\chi(D)f \  \forall D\in D(G/K)\}
\end{equation*}
Let $T_{\chi}$ denote the natural representation of $G$ on this eigenspace,
\begin{equation*}
(T_{\chi}(g)f)(x)=f(g^{-1}\cdot x)
\end{equation*}
for $g\in G, f\in E_{\chi},x\in G/K. $
\begin{lemma}\label{twist}
Fourier transform gives a $K$-equivalent from the space $D(\mathbb{C}^{n})$ of constant coefficient differential operators on $\mathbb{R}^{n}$
onto the space $\mathbb{C}[x_1,\cdots,x_n]$. In particular, it gives an isomorphism of $D(G/K)$ onto $\mathbb{C}[x_1,\cdots,x_n]^{K}$.
\end{lemma}
\begin{lemma}\label{twist}
The homomorphisms of $\mathbb{C}[x_1,\cdots,x_n]^{K}$ into $\mathbb{C}$ are precisely $\chi_{\lambda}:P\mapsto P(\lambda)$,
where $\lambda$ is some element in $\mathbb{C}^{n}$.
\end{lemma}
From Lemma 2.3[J] and Lemma 2.4[H3], we obtain for any $P(\frac{\partial}{\partial x_1},\cdots ,\frac{\partial}{\partial x_n})\in D(G/K)$, where $P(x_{1},\cdots ,x_{n})\in \mathbb{C}[x_1,\cdots,x_n]$ any $\chi:D(G/K)\rightarrow \mathbb{C}$, there exists a unique
\begin{equation*}
\lambda=\left(
          \begin{array}{c}
            \lambda_{1} \\
            \cdot \\
            \cdot \\
            \cdot \\

            \lambda_{n} \\
          \end{array}
        \right)
\end{equation*}
s.t.
\begin{equation*}
\chi(P(\frac{\partial}{\partial x_1},\cdots,\frac{\partial}{\partial x_n}))=P(\lambda_1,\cdots,\lambda_n)
\end{equation*}
Therefore, for any homomorphism $\chi:D(G/K)\rightarrow \mathbb{C}$, we have
\begin{equation*}
E_{\chi}=\{f\in C^{\infty}(G/K)\mid P(\frac{\partial}{\partial x_1},\cdots,\frac{\partial}{\partial x_n})f=P(\lambda_1,\cdots,\lambda_n)f,
\forall P(x_1,\cdots,x_n)\in \mathbb{C}[x_1,\cdots,x_n]^{K}\}
\end{equation*}
For simplicity, we denote $E_\chi=E_\lambda$ from now on.
\end{definition}
\begin{definition}
Generic Property

Let $\chi$ be any character on $\mathbb{R}^{n}$, then there exists
\begin{equation*}
\lambda=\left(
          \begin{array}{c}
            i\lambda_{1} \\
            \cdot \\
            \cdot \\
            \cdot \\

            i\lambda_{n} \\
          \end{array}
        \right)\in \mathbb{C}^{n}
\end{equation*}

where $\lambda_{i}\in \mathbb{R}$,$i=1,\cdots,n$
s.t.
\begin{equation*}
\chi(x)=e^{-\lambda^{T}\cdot x} \ for \ any \
x=\left(
          \begin{array}{c}
            x_{1} \\
            \cdot \\
            \cdot \\
            \cdot \\

            x_{n} \\
          \end{array}
        \right)\in \mathbb{R}^{n}
\end{equation*}
We say $\chi(\ or \ \lambda)$ is generic if for any $g\neq e \ g\in K,g\cdot \lambda\neq \lambda$ holds.

We say $(T_{\lambda},E_{\lambda})$ satisfies the generic property if the same condition holds.
\end{definition}
\smallskip

\textbf{Mackey little group method}

If $\psi$ is an irreducible representation of $G=\mathbb{R}^{n} \rtimes K$, then it can be constructed(up to equivalence)
as follows. If $\chi$ is a character on $\mathbb{R}^{n}$, let $K_{\chi}$denote its $K$-normalizer, so $G_{\chi}=\mathbb{R}^{n} \rtimes K_{\chi}$
is the $G$-normalizer of $\chi$. Write $\widetilde{\chi}$ for the extension of $\chi$ to $G_{\chi}$ given by
$\widetilde{\chi}((x,k))=\chi(x)$; it is a well defined character on $G_{\chi}$. If $\gamma$ is an irreducible unitary representation
of $K_{\chi}$, let $\widetilde{\gamma}$ denote its extension of $G_{\chi}$ given by $\widetilde{\gamma}((x,k))=\gamma(k)$.
Denote $\psi_{\chi,\gamma}=Ind_{G_{\chi}}^{G}(\widetilde{\chi}\bigotimes \widetilde{\gamma})$.
Then there exists choices of $\chi$ and $\gamma$ such that $\psi=\psi_{\chi,\gamma}$.

We say $\psi_{\chi,\gamma}$ satisfies the generic property if $\chi$ is generic.

%
\section{The monomorphism}
In this section, we first introduce some lemmas, and then we use these lemmas to show there exists a monomorphism
from the induced representations to some eigenspace representations both of which satisfy the generic property.
\begin{lemma}\label{twist}
In the notation of Section 2, $\psi_{\chi,\gamma}$ has a $K$-fixed vector if and only if $\gamma$ is the trivial
1-dimensional representation of $K_{\chi}$. In that case, $\psi_{\chi,\gamma}=Ind_{G_{\chi}}^{G}(\widetilde{\chi})$ \ and
the $K$-fixed vector is given(up to scalar multiple)by $u((x,k))=e^{\lambda^{T}\cdot k^{T}\cdot x}$, if $\chi(x)=e^{-\lambda^{T}\cdot x}$.
\end{lemma}
\begin{proof}
The representation space $H_{\psi}$ of $\psi=\psi_{\chi,\gamma}$ consists of all $L^{2}$ functions $f:G\rightarrow H_{\gamma}$
such that $f(g^{'}\cdot (x^{'},k^{'}))=\gamma(k^{'})^{-1}\cdot \chi(x^{'})^{-1}\cdot f(g^{'})$\ for $g^{'}\in G,x^{'}\in \mathbb{R}^{n}$
and $k^{'}\in K_{\chi}$, and $\psi$ acts by $(\psi(g)f)(g^{'})=f(g^{-1}\cdot g^{'})$.

Now suppose that $0\neq f\in H_{\psi}$ is fixed under $\psi(k)$. If $k^{'}\in K_{\chi}$, then $\gamma(k^{'})\cdot f(1)=f(1)$.
If $f(1)=0$, then $f(G_{\chi})=0$ and $K$-invariance says $f=0$, contrary to assumption. Thus $f(1)\neq 0$ and irreducibility of
$\gamma$ forces $\gamma$ to be trivial.

Conversely, if $\gamma$ is trivial and $\chi(x)=e^{-\lambda^{T}\cdot x}$, then $f((x,k))=e^{\lambda^{T}\cdot k^{T}\cdot x}$ is
a nonzero $K$-fixed vector in $H_{\psi}$. And it is the only one, up to scalar multiple, because any two $K$-fixed vectors must
be proportional.
\end{proof}
\begin{lemma}\label{twist}
In the notation of Section 2, $Ind_{G_{\chi}}^{G}(\widetilde{\chi})$ is equivalent to the subrepresentation of $Ind_{\mathbb{R}^{n}}^{G}(\chi)$, which is generated by the
$K$-fixed vector $u((x,k))=e^{\lambda^{T}\cdot k^{T}\cdot x}$. Here $\chi(x)=e^{-\lambda^{T}\cdot x}$. In addtion, both of these two representations are
irreducible.
\end{lemma}
\begin{remark}
If the induced representation satisfies the generic proporty, then $K_{\chi}=\{e\}$. From Lemma 3.1, 3.2[J], we obtain
$\psi_{\chi,\gamma}=Ind_{G_{\chi}}^{G}(\widetilde{\chi})$ is equivalent to $Ind_{\mathbb{R}^{n}}^{G}(\chi)$, which is generated by
the $K$-fixed vector $u((x,k))=e^{\lambda^{T}\cdot k^{T}\cdot x}$.
\end{remark}
\begin{theorem}\label{equal}
Let $(\Phi,V)=Ind_{\mathbb{R}^{n}}^{G}(\chi)$, which satisfies the generic property. Then $(\Phi^{c},V^{*})$ is generated by $u^{*}=(\ ,u)$, where $(f_{1},f_{2})=\int_{G}f_{1}(g)\overline{f_{2}(g)}dg$
for any $f_{1},f_{2}\in V$ and $u((x,k))=e^{\lambda^{T}\cdot k^{T}\cdot x}$ is the $K$ fixed vector in $V$, if $\chi(x)=e^{-\lambda^{T}\cdot x}$.
\end{theorem}
\begin{proof}
The representation space $V$ of $Ind_{\mathbb{R}^{n}}^{G}(\chi)$ consists of all $L^{2}$ functions $f:G\rightarrow V$ such that
$f(g^{'}\cdot (x^{'},e))=\chi(x^{'})^{-1}\cdot f(g^{'})$\ for $g^{'}\in G,x^{'}\in \mathbb{R}^{n}$. Therefore,
$(f_{1},f_{2})=\int_{G}f_{1}(g)\overline{f_{2}(g)}dg$ for any $f_{1},f_{2}\in V$ is a inner product, which satisfies
\begin{equation}\label{l-invariant elements}
\begin{split}
&(\Phi(g_{0})f_{1},\Phi(g_{0})f_{2})=\int_{G}\Phi(g_{0})f_{1}(g)\overline{\Phi(g_{0})f_{2}(g)}dg\\
&=\int_{G}f_{1}(g_{0}^{-1}\cdot g)\overline{f_{2}(g_{0}^{-1}\cdot g)}d(g_{0}^{-1}\cdot g)=\int_{G}f_{1}(g)\overline{f_{2}(g)}dg=(f_{1},f_{2}).
\end{split}
\end{equation}
for any $f_{1},f_{2}\in V$, $g_{0}\in G$.
On one hand, let $v_{1},\cdots ,v_{n}$ be an orthonormal basis of $V$. Then denote $v_{i}^{*}=(\ ,v_{i})\in V^{*}$. It's
easy to check $v_{1}^{*},\cdots , v_{n}^{*}$ is a basis of $V^{*}$. Let $E:V\rightarrow V^{*}$ s.t. $v_{i}\mapsto v_{i}^{*}$.
be a linear map. Then $E$ is isomorphic as a linear map. From Remark 3.3, we know $V$ is irreducible and $V$ is generated by $u((x,k))$.
Let $v_{k}=\sum_{i}l_{ki}\Phi(g_{i})u$,$\Phi(g_{i})u=\sum_{j}a_{ij}v_{j}$.
Then for $\forall v\in V$, we have
 \begin{equation}\label{l-invariant elements}
\begin{split}
&v=\sum_{k}m_{k}v_{k}=\sum_{k}m_{k}(\sum_{i}l_{ki}\Phi(g_{i})u)\\
&=\sum_{i}(\sum_{k}m_{k}l_{ki})\Phi(g_{i})u\\
&=\sum_{i}n_{i}\Phi(g_{i})u=\sum_{i}n_{i}(\sum_{j}a_{ij}v_{j}).
\end{split}
\end{equation}
Therefore, the element of $V$ is of the form $\sum_{i}n_{i}(\sum_{j}a_{ij}v_{j})$ \ $\forall n_{i}\in \mathbb{C}$.
Since $E$ is bijective, we have
\begin{equation}\label{l-invariant elements}
\begin{split}
&V^{*}=\{\sum_{i}n_{i}(\sum_{j}a_{ij}v_{j}^{*}),n_{i}\in \mathbb{C}\}\\
&=\{\sum_{i}\overline{n_{i}}(\sum_{j}\overline{a_{ij}}v_{j}^{*}),n_{i}\in \mathbb{C}\}
\end{split}
\end{equation}
On the other hand, if $\Phi(g)u=\sum_{i}k_{i}v_{i}$, we have $\Phi^{c}(g)u^{*}=\Sigma_{i}\overline{k_{i}}v_{i}^{*}$\label{l-invariant elements}

This can be proved as follows:
If we let $u=\sum_{i}m_{i}v_{i}$,and let
$u^{*}=\sum_{i}\overline{m_{i}}v_{i}^{*}$

Then we have the following result

\begin{equation}\label{l-invariant elements}
\begin{split}
&(u^{*},v)=\sum_{i}\overline{m_{i}}v_{i}^{*}(v)\\
&=\sum_{i}\overline{m_{i}}(v,v_{i})\\
&=(v,\sum_{i}m_{i}v_{i})=(v,u)
\end{split}
\end{equation}

Therefore
\begin{equation}\label{l-invariant elements}
\begin{split}
&(\Phi^{c}(g)u^{*},v)=(u^{*},\Phi(g^{-1})v)\\
&=(\Phi(g^{-1})v,u)=(v,\Phi(g)u)\\
&=(v,\sum_{i}k_{i}v_{i})=\sum_{i}\overline{k_{i}}v_{i}^{*}(v)
\end{split}
\end{equation}

Combining (3.7) and (3.9), we obtain
\begin{equation}\label{l-invariant elements}
\begin{split}
&V^{*}=\{\sum_{i}\overline{n_{i}}(\sum_{j}\overline{a_{ij}}v_{i}^{*}),n_{i}\in \mathbb{C}\}\\
&=\{\sum_{i}\overline{n_{i}}\Phi^{c}(g_{i})u^{*},n_{i}\in \mathbb{C}\}.
\end{split}
\end{equation}
$\therefore u^{*}$ generates $V^{*}$.
\end{proof}
\begin{theorem}\label{equal}
For any $(T_{\lambda},E_{\lambda})$, which satisfies the generic property, we can find a
representation $Ind_{\mathbb{R}^{n}}^{G}(\chi)=(\Phi,V)$, which satisfies the same condition such that there exists
monomorphism from $Ind_{\mathbb{R}^{n}}^{G}(\chi)$ to $(T_{\lambda},E_{\lambda})$. Here $G=\mathbb{R}^{n}\rtimes K$.
\end{theorem}
\begin{proof}
For any $(T_{\lambda},E_{\lambda})$, we consider $(\Phi,V)=Ind_{\mathbb{R}^{n}}^{G}(\chi)$, where $\chi(x)=e^{-\lambda^{T}\cdot x}$. Let
$H=\{f_{v}(g)=\langle \Phi^{c}(g)u^{*},v\rangle \ \forall v\in V\}$. It's easy to get $(\pi,H)$ is a representation of $G$, where
$\pi(g_{0})f(g)=f(g_{0}^{-1}\cdot g)$ for any $g_{0}\in G$.

On one hand, we show $Ind_{\mathbb{R}^{n}}^{\mathbb{R}^{n} \rtimes K}(\chi)$ is equivalent to $(\pi,H)$. Let $F:V\rightarrow H$ by $v\mapsto f_{v}(g)=(\Phi^{c}(g)u^{*},v)$. It's easy
to get $F$ is linear and surjective. Next, we prove $E$ is injective. Suppose $f_{v}(g)=0 \ \forall g\in G$. From Theorem 3.4, we know $u^{*}$ generates $V^{*}$.
If $v\neq 0$, we can find  a $v^{*}\in V^{*}$ s.t. $(v^{*},v)=1$. However, from $(\Phi^{c}(g)u^{*},v)=0 \ for \ \forall g\in G$, we get $(v^{*},v)=0$ a contradiction. $\therefore v=0$
Therefore, $F$ is bijective, which implies $F$ is isomorphism as a linear map.

Next,for any $h\in G,v\in V$
\begin{equation}\label{l-invariant elements}
\begin{split}
&F\circ \Phi(h)v=(\Phi^{c}(g)u^{*},\Phi(h)v)\\
&=(\Phi^{c}(h^{-1})\Phi^{c}(g)u^{*},v)=(\Phi^{c}(h^{-1}g)u^{*},v)\\
&=f_{v}(h^{-1}g)=\pi(h)f_{v}(g)=\pi(h)\circ F\circ v.
\end{split}
\end{equation}
$\therefore F\circ \Phi(h)=\pi(h)\circ F$ \ for any $h\in G$. Therefore, $Ind_{\mathbb{R}^{n}}^{\mathbb{R}^{n}\rtimes K}(\chi)$ is equivalent to $(\pi,H)$.

On the other hand,we show $(\pi,H)$ is a subrepresentation of $(T_{\lambda},E_{\lambda})$.

Firstly, $f_{v}(g)$ is smooth. And for any $k\in K,g\in G$,
\begin{equation}\label{l-invariant elements}
\begin{split}
&f_{v}(gk)=(\Phi^{c}(gk)u^{*},v)=(\Phi^{c}(g)\Phi^{c}(k)u^{*},v)\\
&=(\Phi^{c}(g)u^{*},v)=f_{v}(g).
\end{split}
\end{equation}
$\therefore f_{v}(gk)=f_{v}(g)$ For any $g\in G,k\in K$.
$\therefore f_{v}(g)$ is $K$-right invariant smooth function of $G$.
We can regard $f_{v}(g)$ as a function, which belongs to $C^{\infty}(G/K)$.

Secondly, for any $g=(x_{0},k_{0})=(x_{0},e)(0,k_{0})$
\begin{equation}\label{l-invariant elements}
\begin{split}
&f_{v}(g)=(\Phi^{c}(g)u^{*},v)\\
&=(\Phi^{c}((x_{0},e)(0,k_{0})u^{*},v)=(\Phi^{c}((x_{0},e))u^{*},v)\\
&=(u^{*},\Phi((x_{0},e)^{-1})v)=(u^{*},e^{\lambda^{T}\cdot k^{T}\cdot x_{0}}v).
\end{split}
\end{equation}

For any $P(\frac{\partial}{\partial x_{1}},\cdots ,\frac{\partial}{\partial x_{n}})\in D(G/K)$, from Lemma 2.3, we have
$P(x_{1},\cdots ,x_{n})\in \mathbb{C}[x_{1},\cdots x_{n}]^{K}.$

From K acts on $\mathbb{R}^{n}$ by $K\times \mathbb{R}^{n}\rightarrow \mathbb{R}^{n} \ (k,x)\mapsto k\cdot x$
and $(\Psi(k)e^{*})(e)=e^{*}(k^{-1}\cdot e)$ \ for any $e\in \mathbb{R}^{n},e^{*}\in (\mathbb{R}^{n})^{*} and \  k\in K$.

We obtain the induced action $\Psi: K\times (\mathbb{R}^{n})^{*}\rightarrow (\mathbb{R}^{n})^{*}$

by \ $(k,\left(
                                                                                \begin{array}{c}
                                                                                  x_{1} \\
                                                                                   \cdot\\
                                                                                  \cdot \\
                                                                                  \cdot \\
                                                                                  x_{n} \\
                                                                                \end{array}
                                                                              \right))$ $\rightarrow k^{T}\cdot \left(
                                                                                                                 \begin{array}{c}
                                                                                                                   x_{1} \\
                                                                                                                   \cdot \\
                                                                                                                   \cdot \\
                                                                                                                   \cdot\\
                                                                                                                   x_{n} \\
                                                                                                                 \end{array}
                                                                                                               \right)$
If we denote $k^{T}\cdot \left(
                           \begin{array}{c}
                             x_{1} \\
                             \cdot \\
                             \cdot \\
                             \cdot \\
                             x_{n} \\
                           \end{array}
                         \right)=\left(
                                    \begin{array}{c}
                                       \Psi(k)x_{1} \\
                                      \cdot \\
                                      \cdot \\
                                      \cdot\\
                                     \Psi(k)x_{n} \\
                                    \end{array}
                                  \right)$,
then we have $P(\Psi(k)x_{1},\cdots ,\Psi(k)x_{n})=P(x_{1},\cdots ,x_{n}) \ \forall k\in K$.
Note that $k\cdot k^{T}=I$. Therefore
\begin{equation}\label{l-invariant elements}
\begin{split}
&P(\frac{\partial}{\partial x_{1}},\cdots ,\frac{\partial}{\partial x_{n}})f_{v}(g)=(u^{*},(P(\frac{\partial}{\partial x_{1}},\cdots ,\frac{\partial}{\partial x_{n}})e^{\lambda^{T}\cdot k^{T}\cdot x_{0}})v)\\
&=(u^{*},P(\Psi(k^{T})\lambda_{1},\cdots ,\Psi(k^{T})\lambda_{n})e^{\lambda^{T}\cdot k^{T}\cdot x_{0}}v)\\
&=(u^{*},P(\Psi(k^{-1})\lambda_{1},\cdots ,\Psi(k^{-1})\lambda_{n})e^{\lambda^{T}\cdot k^{T}\cdot x_{0}}v)\\
&=P(\lambda_{1},\cdots ,\lambda_{n})(u^{*},e^{\lambda^{T}\cdot k^{T}\cdot x_{0}}v)\\
&=P(\lambda_{1},\cdots ,\lambda_{n})f_{v}(g).
\end{split}
\end{equation}
For any $P(\frac{\partial}{\partial x_{1}},\cdots ,\frac{\partial}{\partial x_{n}})\in D(G/K)$
$\therefore f_{v}(g)\in E_{\lambda}$, which implies $H\subseteq E_{\lambda}$. This proves the theorem.
\end{proof}

\begin{remark}
From the construction of $Ind_{\mathbb{R}^{n}}^{\mathbb{R}^{n}\rtimes K}(\chi)$, we can see if $Ind_{\mathbb{R}^{n}}^{\mathbb{R}^{n}\rtimes K}(\chi)$
satisfies the generic property then its correspondence $(T_{\lambda},E_{\lambda})$ satisfies the same property. What's more,
once $(T_{\lambda},E_{\lambda})$ is fixed, then $\lambda$ is unique, thus $Ind_{\mathbb{R}^{n}}^{\mathbb{R}^{n}\rtimes K}(\chi)$ is unique. Therefore,
if we let $A=\{$the irreducible representation $Ind_{\mathbb{R}^{n}}^{\mathbb{R}^{n}\rtimes K}(\chi)$, which satisfies the generic property$\}$;
$B=\{$Eigenspace Representation $(T_{\lambda},E_{\lambda})$, which satisfies the generic property$\}$, there is a one-one
correspondence between $A$ and $B$.
\end{remark}

\begin{remark}
For any finite pseudo-reflection group $K$, any $\chi$ of $\mathbb{R}^{n}$ which is generic, there is a
simple fact:
\begin{equation}\label{l-invariant elements}
Ind_{\mathbb{R}^{n}}^{\mathbb{R}^{n}\rtimes K}(\chi)=\bigoplus_{k\in K}\chi^{k}
\end{equation}
Therefore dim $Ind_{\mathbb{R}^{n}}^{\mathbb{R}^{n}\rtimes K}(\chi)=\mid K\mid$
\end{remark}

\section{Irreducibility for eigenspace representations}

In this section  we first recall the invariant theory about finite groups and then use it to show
the eigenspace representations which satisfy the generic property
are irreducible.

\textbf{Notations}
Let $K$ be any finite pseudo-reflection group and let $I(\mathbb{R}^{n})$ denote the set of
$K$-invariants in $S(\mathbb{R}^{n})$ and $I_{+}(\mathbb{R}^{n})\subset I(\mathbb{R}^{n})$ the set of $K$-invariants
without constant term. Similarly, we define $I_{+}((\mathbb{R}^{n})^{*})\subset I((\mathbb{R}^{n})^{*})\subset S((\mathbb{R}^{n})^{*})$.
An element $h\in S^{c}((\mathbb{R}^{n})^{*})$ is said to be $K$-harmonic if $\partial(J)h=0$ \ for
all $J\in I_{+}(\mathbb{R}^{n})$. Here
\begin{equation}\label{l-invariant elements}
(\partial(X)f)(Y)=(\frac{d}{dt}f(Y+tX))_{t=0},f\in \varepsilon(\mathbb{R}^{n});X,Y\in \mathbb{R}^{n}.
\end{equation}

Let $H^{c}((\mathbb{R}^{n})^{*})$ denote the set of $K$-harmonic polynomial functions.Put $H((\mathbb{R}^{n})^{*})=S((\mathbb{R}^{n})^{*})\bigcap H^{c}((\mathbb{R}^{n})^{*})$. For simplicity, we let
$S=S^{c}((\mathbb{R}^{n})^{*})$,$I=I^{c}((\mathbb{R}^{n})^{*})$, $H=H^{c}((\mathbb{R}^{n})^{*})$, $I_{+}\subset I$ be the space
of invariants without constant term, $J$ the ideal $I_{+}S$.
\begin{lemma}\label{twist}
Let $K$ be a compact group of linear transformations of $\mathbb{R}^{n}$ over $R$. Then $S((\mathbb{R}^{n})^{*})=I((\mathbb{R}^{n})^{*})\cdot H((\mathbb{R}^{n})^{*})$ and therefore
$S^{c}((\mathbb{R}^{n})^{*})=I^{c}((\mathbb{R}^{n})^{*})H^{c}((\mathbb{R}^{n})^{*})$.
That is, each polynomial $P$ on $\mathbb{R}^{n}$ has the form $p=\sum_{k}i_{k}h_{k}$, where $i_{k}$ is $G$-invariant and $h_{k}$ harmonic.
\end{lemma}

\begin{lemma}\label{twist}
Let $K$ be a finite pseudo-reflection group acting on the n-dimensional real vector space $\mathbb{R}^{n}$. Then the algebra $I^{c}((\mathbb{R}^{n})^{*})$
of invariants is generated by n homogeneous elements, which are algebraically independent.
\end{lemma}
\begin{lemma}\label{twist}
Let $K$ be a finite pseudo-reflection group, then we obtain $\sum_{0}^{\infty}dim(I^{k})t^{k}=\frac{1}{\mid G\mid}\sum_{g\in K}(det(I-tg))^{-1}$,
where $I$ is the identity operator on $(\mathbb{R}^{n})^{c}$
and $I^{k}$ denotes the subspaces consisting of homogeneous elements of degree $k$.
\end{lemma}
\begin{lemma}\label{twist}
Let $K$ be a finite pseudo-reflection group acting on the real vector space $\mathbb{R}^{n}$. Let $j_{1},\cdots ,j_{n}$ be homogeneous generators for the
algebra $I^{c}((\mathbb{R}^{n})^{*})$ of $K$-invariants. Let $d_{1},\cdots ,d_{n}$be their respective degrees. Then
\begin{equation}\label{l-invariant elements}
\Pi_{i=1}^{n}d_{i}=\mid K\mid.
\end{equation}
Where $\mid K\mid$ denotes the order of $K$.
\end{lemma}
\begin{proof}
By general theory(see[S2],PP.100), we have $trdeg(Q/C)=trdeg(Q/K)+trdeg(K/C)$. And we observe from it the $j_{i}$ are algebraically independent. Let
$j\in I^{k}$, the space of homogeneous invariants of degree k. Then $j$ is a linear combination of monomials $j_{1}^{a_{1}}\cdots j_{n}^{a_{n}}$
for which
\begin{equation}\label{l-invariant elements}
 a_{1}d_{1}+\cdots +a_{n}d_{n}=k.
\end{equation}
By the algebraic independence of the $j_{i}$, it follows that dim($I^{k}$) equals the number of nonnegative integral solutions $(a_{1},\cdots ,a_{n})$
to (4.7). Hence
\begin{equation}\label{l-invariant elements}
\sum_{0}^{\infty}dim(I^{k})t^{k}=(1-t^{d_{1}})^{-1}\cdots  (1-t^{d_{n}})^{-1}
\end{equation}
Combining this with Lemma 4.4[H3], we conclude
\begin{equation}\label{l-invariant elements}
 \mid K\mid\cdot \prod_{1\leq i\leq n}(1+t+\cdots +t^{d_{i}-1})^{-1}=\sum_{\sigma \in K}\cdot \prod_{1\leq j\leq n}\frac{1-t}{1-tC_{\sigma_{j}}}
\end{equation}
If the $C_{\sigma_{j}}$ are the eigenvalues of $\sigma$ counted with multiplicity. Letting $t\rightarrow 1$, only the term $\sigma=1$ on the right
gives a contribution, so we obtain the formula of the lemma.
\end{proof}
\begin{lemma}\label{twist}
Let $j_{1}\cdots j_{m}\in I$ such that $j_{1}\notin\sum_{2}^{m}j_{s}I$. If $q_{1}\cdots q_{m}\in S$ are homogeneous elements such that $\sum_{1}^{m}j_{s}q_{s}=0$,
then $q_{1}\in J$.
\end{lemma}
\begin{theorem}\label{equal}
Let $K$ be a finite pseudo-reflection group and let the notation be as above. Then dim$H$=$\mid K\mid$ and the mapping $\phi:j\bigotimes h\rightarrow jh$
extends to a linear bijection of $I\bigotimes H$ onto $S$. Moreover,
\begin{equation}\label{l-invariant elements}
 \sum_{k\geq 0}(dimH^{k})t^{k}=\prod_{1\leq i\leq n}(1+t+\cdots +t^{d_{i}-1})
\end{equation}
\end{theorem}
\begin{proof}
We know from Lemma 4.2[H3] that $\phi$ is surjective. To prove that it is injective, we must show that if $\Sigma_{r,s}a_{r,s}i_{r}h_{s}=0$, where $a_{r,s}\in \mathbb{C}$
and $\{i_{r}\}$and $\{h_{s}\}$ are homogeneous bases of the vector spaces $I$ and $H$, respectively, then $a_{r,s}=0$. We write the relation in the form
\begin{equation}\label{l-invariant elements}
 \sum_{s}h_{s}(\sum_{r}a_{r,s}i_{r})=0.
\end{equation}
Put \ $I_{s}=\sum_{r}a_{r,s}i_{r}$. We have to prove that each $I_{s}=0$ and for this it suffices to consider the case in which each $I_{s}$ is homogeneous and deg$h_{s}+degI_{s}$ the same
for all $s$. Suppose there were an $I_{s}\neq 0$. We write it in the form
\begin{equation}\label{l-invariant elements}
 I_{s}=\sum a_{m_{1},\cdots m_{n},s}j_{1}^{m_{1}}\cdots j_{n}^{m_{n}}
\end{equation}
with nonzero coefficients $a_{m_{1},\cdots m_{n},s}$. Then
\begin{equation}\label{l-invariant elements}
 \sum_{(m)}(\sum_{s}a_{m_{1},\cdots m_{n},s}h_{s})j_{1}^{m_{1}}\cdots j_{n}^{m_{n}}=0,
\end{equation}
and at least one of the monomials
$j_{1}^{m_{1}}\cdots j_{n}^{m_{n}}$ is not in the ideal in $I$ generated by the others. The corresponding term
\begin{equation}\label{l-invariant elements}
\sum_{s}a_{m_{1},\cdots m_{n},s}h_{s}
\end{equation}
then belongs to
$J$ according to Lemma 4.10[H3]. But by $S^{k}((\mathbb{R}^{n})^{*})=(I_{+}((\mathbb{R}^{n})^{*})S((\mathbb{R}^{n})^{*}))^{k}+H^{k}((\mathbb{R}^{n})^{*})$, this term will have to vanish and then the linear independence of
the $h_{s}$ gives the contradiction $a_{m_{1},\cdots m_{n},s}=0$.

The identification $I\bigotimes H=S$ implies the identity
\begin{equation}\label{l-invariant elements}
 \sum_{k\geq 0}(dimI^{k})t^{k}\sum_{l\geq 0}(dimH^{l})t^{l}=\sum_{m\geq 0}(dimS^{m})t^{m}.
\end{equation}
Since the right-hand side equals $(1-t)^{-n}$ the formula for $(dimH^{l})t^{l}$ follows from (4.8). Putting $t=1$, we obtain the formula dim$H$=$\mid K\mid$ form Lemma 4.5[S].
\end{proof}
\begin{theorem}\label{Main Theorem}
In the notation above, for each $\lambda\in \mathbb{C}^{n}$, we have dim$E_{\lambda}\leq \mid K\mid$. Furthermore, for the case when $\lambda$ is generic, $(T_{\lambda},E_{\lambda})$ is irreducible.
If we let $A=\{Ind_{R^{n}}^{R^{n}\rtimes K}(\chi)$ $\mid \chi$ is generic $\}$,
$B=\{(T_{\lambda},E_{\lambda})\mid E_{\lambda}$ is generic$\}$. Then there exists equivalent one-one correspondence between the elements of $A$ and $B$.
\end{theorem}
\begin{proof}
From the conclusion of Theorem 4.11[H3], dim$H$=$\mid K\mid$. Let $h_{1},\cdots h_{\mid K\mid}$ be a basis of $H$ and let $H_{1},\cdots H_{\mid K\mid}$ be the corresponding
members of $S((\mathbb{R}^{n})^{c}))$. [Under the identification of $(\mathbb{R}^{n})^{c}$ and $((\mathbb{R}^{n})^{c})^{*}$ by means of $B$, which is a nondegenerate symmetric bilinear form on
$(\mathbb{R}^{n})^{c}\times (\mathbb{R}^{n})^{c}]$. Let $f\in E_{\lambda}$ and put for $x_{0}\in R^{n}$
\begin{equation}\label{l-invariant elements}
C_{i}(f)=(\partial(H_{i})f)(x_{0})(1\leq i\leq \mid K\mid).
\end{equation}
Since $S=IH$, it is clear that if $C_{1}(f)=\cdots C_{\mid K\mid}(f)=0$, then $f=0$. Thus the mapping
\begin{equation}\label{l-invariant elements}
f\mapsto (C_{1}(f),\cdots ,C_{\mid K\mid}(f))
\end{equation}
is one-to-one linear mapping of $E_{\lambda}$ into $\mathbb{C}^{\mid K\mid}$. So dim$E_{\lambda}\leq \mid K\mid$. Thus we have proved the first statement of the theorem.

For the second part, for any $(T_{\lambda},E_{\lambda})$ which satisfies the generic property, from the conclusion of Theorem 3.10, there exists $Ind_{\mathbb{R}^{n}}^{\mathbb{R}^{n}\rtimes K}(\chi)$ which
satisfies the same property and a monomorphism from $Ind_{\mathbb{R}^{n}}^{\mathbb{R}^{n}\rtimes K}(\chi)$ to $(T_{\lambda},E_{\lambda})$. Therefore,
dim$Ind_{\mathbb{R}^{n}}^{\mathbb{R}^{n}\rtimes K}(\chi)\leq$ dim$E_{\lambda}\leq \mid K\mid$. However, according to Remark 3.16, we obtain
dim$Ind_{\mathbb{R}^{n}}^{\mathbb{R}^{n}\rtimes K}(\chi)=\mid K\mid$. Then dim$Ind_{\mathbb{R}^{n}}^{\mathbb{R}^{n}\rtimes K}(\chi)=$ dim$E_{\lambda}= \mid K\mid$.
Therefore $Ind_{\mathbb{R}^{n}}^{\mathbb{R}^{n}\rtimes K}(\chi)\cong (T_{\lambda},E_{\lambda})$, which $(T_{\lambda},E_{\lambda})$ is irreducible.
Furthermore,from Theorem 3.10, Remark 3.15 and the conclusion above, we obtain there exists equivalent one-one correspondence between the elements $A$ and $B$.
\end{proof}

\section{The $D_{n}$ case}

In this section, we use the knowledge of previous sections to calculate explicitly for a particular example when
$G=R^{2}\rtimes D_{n}$, $K=D_{n}$. Here $D_{n}$ is the dihendral group, its elements have the following form:
\begin{equation}\label{l-invariant elements}
R_{k}=\left(
        \begin{array}{cc}
          cos\frac{2\pi k}{n} & -sin\frac{2\pi k}{n} \\
          sin\frac{2\pi k}{n} & cos\frac{2\pi k}{n} \\
        \end{array}
      \right)
\end{equation}
and
\begin{equation}\label{l-invariant elements}
S_{k}=\left(
        \begin{array}{cc}
          cos\frac{2\pi k}{n} & sin\frac{2\pi k}{n} \\
          sin\frac{2\pi k}{n} & -cos\frac{2\pi k}{n} \\
        \end{array}
      \right)
\end{equation}
\textbf{$D_{n}$ acts on $(\mathbb{R}^{2})^{*}$}
Each $g\in D_{n}\subset GL(\mathbb{R}^{2})$ acts on $\mathbb{R}^{2}$ and on $(\mathbb{R}^{2})^{*}$ by standard action and $(\Psi(g)e^{*})(e)=e^{*}(g^{-1}\cdot e)$, respectively. These
actions extend to automorphisms of $S(\mathbb{R}^{2})$ and $S((\mathbb{R}^{2})^{*})$. According to the above definition, we can calculate the action explicitly.

Let
$e_{1}=\left(
         \begin{array}{c}
           1 \\
           0 \\
         \end{array}
       \right)$ \ $e_{2}=\left(
                           \begin{array}{c}
                             0 \\
                             1 \\
                           \end{array}
                         \right)$ \ $x_{i}(e_{j})=\delta_{ij}$ \ $1\leq i,j\leq 2$.

Let $\Psi(R_{k})x_{1}=a_{1}\cdot x_{1}+ a_{2}\cdot x_{2}$, where $\Psi$ is defined in the proof of Theorem 3.10,
Then \begin{equation}\label{l-invariant elements}
\begin{split}
&a_{1}=\Psi(R_{k})x_{1}(e_{1})=x_{1}(R_{k}^{-1}\cdot e_{1})=x_{1}(R_{k}^{T}\cdot e_{1})\\
&=x_{1}(\left(
          \begin{array}{c}
            cos\frac{2\pi k}{n} \\
            -sin\frac{2\pi k}{n} \\
          \end{array}
        \right))=cos\frac{2\pi k}{n}\\
&a_{2}=\Psi(R_{k})x_{1}(e_{2})=x_{1}(R_{k}^{-1}\cdot e_{2})=sin\frac{2\pi k}{n}.
\end{split}
\end{equation}
$\therefore \Psi(R_{k})x_{1}=cos\frac{2\pi k}{n}\cdot x_{1}+sin\frac{2\pi k}{n}\cdot x_{2}. $
Similarly,$\Psi(R_{k})x_{2}=-sin\frac{2\pi k}{n}\cdot x_{1}+cos\frac{2\pi k}{n}\cdot x_{2}. $
$\therefore \Psi(R_{k})\left(
                         \begin{array}{c}
                           x_{1} \\
                           x_{2} \\
                         \end{array}
                       \right)=R_{k}^{T}\cdot \left(
                         \begin{array}{c}
                           x_{1} \\
                           x_{2} \\
                         \end{array}
                       \right)$

Similarly,$\Psi(S_{k})\left(
                         \begin{array}{c}
                           x_{1} \\
                           x_{2} \\
                         \end{array}
                       \right)=S_{k}^{T}\cdot \left(
                         \begin{array}{c}
                           x_{1} \\
                           x_{2} \\
                         \end{array}
                       \right)$
$\forall 0\leq k\leq n-1$.
\begin{lemma}\label{twist}
For a finite group $G\subseteq GL(\mathbb{R}^{n})$, let $\mid G\mid=g$, then $C[x_{1},\cdots ,x_{n}]^{G}$ is generated
by the $\left(
                                                                          \begin{array}{c}
                                                                            g+n \\
                                                                            n \\
                                                                          \end{array}
                                                                        \right)$
polynomials \ $\frac{1}{g}\sum_{M\in G}M\cdot f$, as $f$ ranges over all $\left(
                                                                          \begin{array}{c}
                                                                            g+n \\
                                                                            n \\
                                                                          \end{array}
                                                                        \right)$
monomials in the variables $x_{1},\cdots ,x_{n}$ \ of degree at most $g$.
\end{lemma}
\begin{remark}
The above lemma is in fact showed by Nother(see[ZS],PP.275-276.]. Note that $D_{n}$ is a finite pseudo-reflection group. If we
combine above lemma, the above calculation and Lemma 4.3, we can calculate $\mathbb{C}[x_{1},x_{2}]^{D_{n}}$ explicitly. Let's look at the following
Proposition.
\end{remark}
\begin{prop}
For general $n\geq 3$, we obtain $\mathbb{C}[x_{1},x_{2}]^{D_{n}}=\mathbb{C}[z\cdot \overline{z},z^{n}+\overline{z}^{n}]$, where $z=x_{1}+i\cdot x_{2}$.
\end{prop}
\begin{proof}
We consider $f_{1}=x_{1}^{2}$ and $f_{2}=x_{1}^{n}$. According to Lemma 5.4, $\sum_{M\in D_{n}}M\cdot f_{1}$,$\sum_{M\in D_{n}}M\cdot f_{2}\in \mathbb{C}[x_{1},x_{2}]^{D_{n}}$.

Note that for any $m\in N$, the following equations hold
\begin{equation}\label{l-invariant elements}
\begin{split}
&\sum_{k=0}^{n-1}(cos\frac{2\pi k}{n}\cdot x_{1}+sin\frac{2\pi k}{n}\cdot x_{2})^{m}=\sum_{k=1}^{n}(cos\frac{2\pi k}{n}\cdot x_{1}+sin\frac{2\pi k}{n}\cdot x_{2})^{m}\\
&=\sum_{k=1}^{n}(cos\frac{2\pi(n-k)}{n}\cdot x_{1}+sin\frac{2\pi(n-k)}{n}\cdot x_{2})^{m}=\sum_{k=1}^{n}(cos\frac{2\pi k}{n}\cdot x_{1}-sin\frac{2\pi k}{n}\cdot x_{2})^{m}\\
&=\sum_{k=0}^{n-1}(cos\frac{2\pi k}{n}\cdot x_{1}-sin\frac{2\pi k}{n}\cdot x_{2})^{m}.
\end{split}
\end{equation}
 Therefore, according to the above calculation, we obtain
 \begin{equation}\label{l-invariant elements}
\begin{split}
&\sum_{M\in D_{n}}M\cdot f_{1}=\sum_{k=0}^{n-1}(cos\frac{2\pi k}{n}\cdot x_{1}-sin\frac{2\pi k}{n}\cdot x_{2})^{2}+\sum_{k=0}^{n-1}(cos\frac{2\pi k}{n}\cdot x_{1}+sin\frac{2\pi k}{n}\cdot x_{2})^{2}\\
&=2\sum_{k=0}^{n-1}(cos^{2}\frac{2\pi k}{n}\cdot x_{1}^{2}+sin^{2}\frac{2\pi k}{n}\cdot x_{2}^{2})\\
&=\sum_{k=0}^{n-1}((1+cos\frac{4\pi k}{n})\cdot x_{1}^{2}+(1-cos\frac{4\pi k}{n})\cdot x_{2}^{2}).
\end{split}
\end{equation}
For $n\geq 3$, $\frac{4\pi}{n}\neq 2\pi k$ $\therefore \sum_{k=0}^{n-1}e^{i\cdot \frac{4\pi k}{n}}=0$, which implies $\sum_{k=0}^{n-1}cos\frac{4\pi k}{n}=0$
$\therefore \sum_{M\in D_{n}}M\cdot f_{1}=n\cdot (x_{1}^{2}+x_{2}^{2})$
\begin{equation}\label{l-invariant elements}
\therefore z\cdot \overline{z}=(x_{1}+i\cdot x_{2})(x_{1}-i\cdot x_{2})=x_{1}^{2}+x_{2}^{2}\in \mathbb{C}[x_{1},x_{2}]^{D_{n}}
\end{equation}

Next, we prove $z^{n}+\overline{z}^{n}\in \mathbb{C}[x_{1},x_{2}]^{D_{n}}$. The general idea is as follows:

Firstly, we compute when $n=2m+1$, $z^{n}+\overline{z}^{n}\in \mathbb{C}[x_{1},x_{2}]^{D_{n}}$. What's more, using the same method, we can compute when $n=2m,z^{2m}+\mid z\mid^{2m}+\overline{z}^{2m}\in \mathbb{C}[x_{1},x_{2}]^{D_{n}}$.
Since $\mid z\mid^{2m}=(x_{1}^{2}+x_{2}^{2})^{m}$, $z^{2m}+\overline{z}^{2m}\in \mathbb{C}[x_{1},x_{2}]^{D_{n}}$.

Now, we give the concrete proof as follows:

When $n=2m+1$, let $z=ae^{i\theta}=x_{1}+i\cdot x_{2}$. Therefore, $\mid z\mid=\sqrt{x_{1}^{2}+x_{2}^{2}}$.
\begin{equation}\label{l-invariant elements}
\begin{split}
&\sum_{M\in D_{n}}M\cdot f_{2}\\
&=\sum_{k=0}^{n-1}((cos^{2}\frac{2\pi k}{n}\cdot x_{1}-sin^{2}\frac{2\pi k}{n}\cdot x_{2})^{2m+1}+(cos^{2}\frac{2\pi k}{n}\cdot x_{1}+sin^{2}\frac{2\pi k}{n}\cdot x_{2})^{2m+1})\\
&=\sum_{k=0}^{n-1}((\frac{e^{i\cdot \frac{2\pi k}{n}z}+e^{-i\cdot \frac{2\pi k}{n}\overline{z}}}{2})^{2m+1}+(\frac{e^{-i\cdot \frac{2\pi k}{n}z}+e^{i\cdot \frac{2\pi k}{n}\overline{z}}}{2})^{2m+1})\\
&=\sum_{k=0}^{n-1}(( \frac{a\cdot e^{i\cdot (\theta+\frac{2\pi k}{n})}+a\cdot e^{-i\cdot (\theta+\frac{2\pi k}{n})}}{2})^{2m+1}+(\frac{a\cdot e^{i\cdot (\theta-\frac{2\pi k}{n})}+a\cdot e^{-i\cdot (\theta-\frac{2\pi k}{n})}}{2})^{2m+1})\\
&=\sum_{k=0}^{n-1}(\sum_{l=0}^{2m+1}\frac{C_{2m+1}^{l}a^{2m+1-l}e^{i(\theta+\frac{2\pi k}{n})(2m+1-l)}a^{l}e^{-i(\theta+\frac{2\pi k}{n})}}{2^{2m+1}}\\
&+\sum_{l=0}^{2m+1}\frac{C_{2m+1}^{l}a^{2m+1-l}e^{i(\theta-\frac{2\pi k}{n})(2m+1-l)}a^{l}e^{-i(\theta-\frac{2\pi k}{n})}}{2^{2m+1}})\\
&=\sum_{k=0}^{n-1}(\sum_{l=0}^{2m+1}\frac{C_{2m+1}^{l}a^{2m+1}e^{i\theta(2m+1-2l)}e^{i\frac{2m+1-2l}{n}2\pi k}}{2^{2m+1}}\\
&+\sum_{l=0}^{2m+1}\frac{C_{2m+1}^{l}a^{2m+1}e^{i\theta(2m+1-2l)}e^{-i\frac{2m+1-2l}{n}2\pi k}}{2^{2m+1}}).
\end{split}
\end{equation}
Note that
\begin{equation}\label{l-invariant elements}
e^{i\frac{2m+1-2l}{n}2\pi}=1\Longleftrightarrow l=2m+1, 0\Longleftrightarrow e^{-i\frac{2m+1-2l}{n}2\pi}=1
\end{equation}
Therefore,
\begin{equation}\label{l-invariant elements}
\begin{split}
&\sum_{M\in D_{n}}M\cdot f_{2}=\frac{1}{2^{2m+1}}(C_{2m+1}^{0}a^{2m+1}e^{i\theta (2m+1)}+C_{2m+1}^{0}a^{2m+1}e^{i\theta (2m+1)}\\
&+C_{2m+1}^{2m+1}a^{2m+1}e^{-i\theta (2m+1)}+C_{2m+1}^{2m+1}a^{2m+1}e^{-i\theta (2m+1)})\\
&=\frac{2}{2^{2m+1}}((ae^{i\theta})^{2m+1}+(ae^{-i\theta})^{2m+1})\\
&=\frac{1}{2^{2m}}(z^{2m+1}+\overline{z}^{2m+1}).
\end{split}
\end{equation}
Therefore, when $n=2m+1,z^{n}+\overline{z}^{n}\in \mathbb{C}[x_{1},x_{2}]^{D_{n}}$.

Similarly, we can show when $n=2m,z^{n}+\overline{z}^{n}\in \mathbb{C}[x_{1},x_{2}]^{D_{n}}$.

Therefore, $z\cdot \overline{z}$ and $z^{n}+\overline{z}^{n} \in \mathbb{C}[x_{1},x_{2}]^{D_{n}}$ for $\forall n\geq 3$. Since
$z\cdot \overline{z}$ and $z^{n}+\overline{z}^{n}$ are algebraiclly independent, according to Lemma 4.3, we
obtain $\mathbb{C}[x_{1},x_{2}]^{D_{n}}=\mathbb{C}[z\cdot \overline{z},z^{n}+\overline{z}^{n}]$.
\end{proof}
\begin{remark}
According to Lemma 2.3, for every $E_{\lambda}$,$\exists \left(
                                                          \begin{array}{c}
                                                            \mu_{1}^{\lambda} \\
                                                            \mu_{2}^{\lambda}\\
                                                          \end{array}
                                                        \right)\in \mathbb{C}^{2}$ \ s.t.
\begin{equation}\label{l-invariant elements}
E_{\lambda}=\{f\in \varepsilon(R^{2})\mid \frac{\partial}{\partial z}\frac{\partial}{\partial \overline{z}}f=\mu_{1}^{\lambda}f \ \
((\frac{\partial}{\partial z})^{n}+(\frac{\partial}{\partial \overline{z}})^{n})f=\mu_{2}^{\lambda}f\}
\end{equation}
\end{remark}
\begin{remark}
Let notation as above,
\begin{equation}\label{l-invariant elements}
\begin{split}
&H=\{f\in \varepsilon(R^{2})\mid \frac{\partial}{\partial z}\frac{\partial}{\partial \overline{z}}f=0 \ \
((\frac{\partial}{\partial z})^{n}+(\frac{\partial}{\partial \overline{z}})^{n})f=0 \}\\
&=\mathbb{C}\bigoplus \mathbb{C}Z\bigoplus \mathbb{C}\overline{Z}\bigoplus \mathbb{C}Z^{2}\bigoplus \mathbb{C}\overline{Z}^{2}\bigoplus\cdots \bigoplus \mathbb{C}Z^{n-1}\bigoplus \mathbb{C}\overline{Z}^{n-1}\bigoplus \mathbb{C}(Z^{n}-\overline{Z}^{n}).
\end{split}
\end{equation}
Therefore, dim$H$=$2n$=$\mid D_{n}\mid$.
\end{remark}
\begin{prop}
Let $\chi(x)=e^{-\lambda^{T}\cdot x}$, where $\lambda=\left(
                                                        \begin{array}{c}
                                                          \lambda_{1} \\
                                                          \lambda_{2} \\
                                                        \end{array}
                                                      \right)$
be generic. Then $Ind_{\mathbb{R}^{2}}^{\mathbb{R}^{2}\rtimes D_{n}}(\chi)$ is
equivalent to $(T_{\lambda},E_{\lambda})$, where $E_{\lambda}=\{f\in \varepsilon(\mathbb{R}^{2})\mid \frac{\partial}{\partial z}\frac{\partial}{\partial \overline{z}}f=\frac{1}{4}(\lambda_{1}^{2}+
\lambda_{2}^{2})f \ \ ((\frac{\partial}{\partial z})^{n}+(\frac{\partial}{\partial \overline{z}})^{n})f=\frac{1}{2^{n}}((\lambda_{1}-i\lambda_{2})^{n}+(\lambda_{1}+i\lambda_{2})^{n})f\}$
Furthermore, $(T_{\lambda},E_{\lambda})$ is irreducible and all eigenspace representations of $\mathbb{R}^{2}\rtimes D_{n}$, which satisfy the generic property is of this form.
\end{prop}
\begin{proof}
According to Theorem 4.14, we obtain $Ind_{\mathbb{R}^{2}}^{\mathbb{R}^{2}\rtimes D_{n}}(\chi)$ is equivalent to $(T_{\lambda},E_{\lambda})$. Therefore $(T_{\lambda},E_{\lambda})$
is irreducible. Note that $\frac{\partial}{\partial z}=\frac{1}{2}\frac{\partial}{\partial x_{1}}-\frac{i}{2}\frac{\partial}{\partial x_{2}}$,
$\frac{\partial}{\partial \overline{z}}=\frac{1}{2}\frac{\partial}{\partial x_{1}}+\frac{i}{2}\frac{\partial}{\partial x_{2}}$.

Therefore, we have
\begin{equation}\label{l-invariant elements}
\begin{split}
&\frac{\partial}{\partial z}\frac{\partial}{\partial \overline{z}}=P_{1}(\frac{\partial}{\partial x_{1}},\frac{\partial}{\partial x_{2}})\\
&=(\frac{1}{2}\frac{\partial}{\partial x_{1}}-\frac{i}{2}\frac{\partial}{\partial x_{2}})(\frac{1}{2}\frac{\partial}{\partial x_{1}}+\frac{i}{2}\frac{\partial}{\partial x_{2}})\\
&=\frac{1}{4}((\frac{\partial}{\partial x_{1}})^{2}+(\frac{\partial}{\partial x_{2}})^{2})\\
&(\frac{\partial}{\partial z})^{n}+(\frac{\partial}{\partial \overline{z}})^{n}=P_{2}(\frac{\partial}{\partial x_{1}},\frac{\partial}{\partial x_{2}})\\
&=\frac{1}{2^{n}}((\frac{\partial}{\partial x_{1}}-i\frac{\partial}{\partial x_{2}})^{n}+(\frac{\partial}{\partial x_{1}}+i\frac{\partial}{\partial x_{2}})^{n}).
\end{split}
\end{equation}

Then we obtain
\begin{equation}\label{l-invariant elements}
\begin{split}
&E_{\lambda}=\{f\in \varepsilon(\mathbb{R}^{2})\mid P(\frac{\partial}{\partial x_{1}},\frac{\partial}{\partial x_{2}})f=
P(\lambda_{1},\lambda_{2})f, \ \forall P(\frac{\partial}{\partial x_{1}},\frac{\partial}{\partial x_{2}})\in D(\mathbb{R}^{2}\rtimes D_{n}/D_{n})\}\\
&=\{f\in \varepsilon(\mathbb{R}^{2})\mid P_{1}(\frac{\partial}{\partial x_{1}},\frac{\partial}{\partial x_{2}})f=
P_{1}(\lambda_{1},\lambda_{2})f,P_{2}(\frac{\partial}{\partial x_{1}},\frac{\partial}{\partial x_{2}})f=P_{2}(\lambda_{1},\lambda_{2})f\}\\
&=\{f\in \varepsilon(\mathbb{R}^{2})\mid \frac{\partial}{\partial z}\frac{\partial}{\partial \overline{z}}f=\frac{1}{4}(\lambda_{1}^{2}+
\lambda_{2}^{2})f \ \ ((\frac{\partial}{\partial z})^{n}+(\frac{\partial}{\partial \overline{z}})^{n})f=\frac{1}{2^{n}}((\lambda_{1}-i\lambda_{2})^{n}+(\lambda_{1}+i\lambda_{2})^{n})f\}.
\end{split}
\end{equation}
Finally, from the conclusion of Remark 3.15, we can get all eigenspace representations of $\mathbb{R}^{2}\rtimes D_{n}$, which satisfy the generic property is of this form.
\end{proof}
\begin{remark}
In this paper, we mainly consider the irreducibility of the eigenspace representations of $G=\mathbb{R}^{n}\rtimes K$, where $K$ is a finite pseudo-reflection group. In fact,
the same problem for the general case when $K$ is any closed subgroup of $O(n)$ is also of great value.
\end{remark}

\end{document}